\documentclass[12pt,reqno]{amsart}

\usepackage{amssymb}

\usepackage[all]{xy}

\numberwithin{equation}{section}

\newtheorem{theorem}{Theorem}[section]
\newtheorem{proposition}[theorem]{Proposition}

\newtheorem{corollary}[theorem]{Corollary}

\theoremstyle{definition}

\newtheorem{remark}[theorem]{Remark}

\begin{document}

\title[Pushforward of structure sheaf and virtual global generation]{Pushforward of 
structure sheaf and virtual global generation}

\author[I. Biswas]{Indranil Biswas}

\address{Department of Mathematics, Shiv Nadar University, NH91, Tehsil
Dadri, Greater Noida, Uttar Pradesh 201314, India}

\email{indranil.biswas@snu.edu.in, indranil29@gmail.com}

\author[M. Kumar]{Manish Kumar}

\address{Statistics and Mathematics Unit, Indian Statistical Institute,
Bangalore 560059, India}

\email{manish@isibang.ac.in}

\author[A.J. Parameswaran]{A. J. Parameswaran}

\address{School of Mathematics, Tata Institute of Fundamental
Research, Homi Bhabha Road, Bombay 400005, India}

\email{param@math.tifr.res.in}



\begin{abstract}
Let $f\,:\,X\,\longrightarrow \,Y$ be a generically smooth morphism between
irreducible smooth projective curves over an algebraically closed field of arbitrary characteristic. We prove
that the vector bundle $((f_*{\mathcal O}_X)/{\mathcal O}_Y)^*$ is virtually globally
generated. Moreover, $((f_*{\mathcal O}_X)/{\mathcal O}_Y)^*$ is ample if and only if $f$
is genuinely ramified.
\end{abstract}

\maketitle

\section{Introduction}

Let $X$ and $Y$ be irreducible smooth projective curves over an algebraically closed field
$k$ --- there is no assumption on the characteristic of $k$ --- and
let $f\,:\,X\,\longrightarrow \,Y$ be a generically smooth morphism.
Then we have ${\mathcal O}_Y\, \subset\, f_*{\mathcal O}_X$. In \cite{BP3} it was shown that
the homomorphism of \'etale fundamental groups
$f_*\, :\, \pi^{\rm et}_1(X)\, \longrightarrow\,\pi^{\rm et}_1(Y)$
induced by $f$ is surjective if and only if ${\mathcal O}_Y$ is the unique maximal
semistable subsheaf of $f_*{\mathcal O}_X$. We call $f$ to be genuinely ramified
if ${\mathcal O}_Y$ is the unique maximal semistable subsheaf of $f_*{\mathcal O}_X$.
On the other hand, $f$ is called primitive if the above homomorphism $f_*$ of \'etale fundamental groups
is surjective \cite{CLV}. So $f$ is genuinely ramified if and only if it is primitive.

The main result of \cite{BP3} says the following: If $f\,:\,X\,\longrightarrow \,Y$ is
genuinely ramified, and $E$ is a stable vector bundle on $Y$, then $f^*E$ is also stable.
This was proved by investigating the quotient bundle $(f_*{\mathcal O}_X)/{\mathcal O}_Y$.

The dual vector bundle $((f_*{\mathcal O}_X)/{\mathcal O}_Y)^*$ is called the
Tschirnhausen bundle for $f$ (see \cite{CLV}). The following is the main result of
\cite{CLV}: Let $f\,:\,X\,\longrightarrow \,Y$ be a general primitive degree $r$ cover,
where ${\rm genus}(X)\,=\, g$ and ${\rm genus}(Y)\,=\, h$,
over an algebraically closed field of characteristic zero or greater than $r$. Then
\begin{enumerate}
\item $((f_*{\mathcal O}_X)/{\mathcal O}_Y)^*$ is semistable if $h\,=\, 1$, and

\item $((f_*{\mathcal O}_X)/{\mathcal O}_Y)^*$ is stable
if $h \,\geq \, 2$.
\end{enumerate}
Note that the above mentioned result of \cite{BP3} can be reformulated as follows: Let 
$f\,:\,X\,\longrightarrow \,Y$ be a generically smooth morphism between irreducible smooth projective curves.
Then $f^*E$ is stable for every stable vector bundle $E$ on $Y$ if and only if
$$
\mu_{\rm min}(((f_*{\mathcal O}_X)/{\mathcal O}_Y)^*) \, >\, 0.
$$
(Recall that $\mu_{\rm min}$ denotes the slope of the smallest quotient \cite[p.~16, Definition 1.3.2]{HL}.)
See \cite{CLV} for more on Tschirnhausen bundles.

A vector bundle on an irreducible smooth projective curve $Z$ is called virtually globally generated if its
pullback, under some surjective morphism to $Z$ from some irreducible smooth projective curve, is generated by its
global sections; see Section \ref{se3}.

We prove the following (see Theorem \ref{thm2}):

\textit{Let $X$ and $Y$ be irreducible smooth projective curves and
$$
f\, :\, X\, \longrightarrow\, Y
$$
a generically smooth morphism. Then $(f_*{\mathcal O}_X)^*$ is virtually globally generated.}

Note that this implies that $((f_*{\mathcal O}_X)/{\mathcal O}_Y)^*$ is
virtually globally generated (see Corollary \ref{cor3}).

In Remark \ref{r-f} it is shown that Corollary \ref{cor3} fails in higher dimensions.

We prove the following (see Corollary \ref{cor2}):

\textit{Let $f\, :\, X\, \longrightarrow\, Y$ be a generically smooth morphism between two
irreducible smooth projective curves. Then $f$ is genuinely ramified if and only if
$((f_*{\mathcal O}_X)/{\mathcal O}_Y)^*$ is ample.}

It may be mentioned that the condition in Theorem \ref{thm2} and Corollary \ref{cor2}
that $f$ is generically smooth is essential. To give an example, take $Y$ to be a
smooth projective curve of genus at least two, and let $F_Y\, :\, Y\, \longrightarrow\, Y$
be the absolute Frobenius morphism of $Y$. Then
$(F_{Y*}{\mathcal O}_Y))/{\mathcal O}_Y$ is in fact ample.

\section{Genuinely ramified maps, direct image and ampleness}

The base field $k$ is assumed to be algebraically closed. For a vector bundle $E$ on an irreducible smooth projective curve
$X$, if
$$
E_1\, \subset\, \cdots\, \subset\, E_{n-1}\, \subset\, E_n\,=\, E
$$
is the Harder--Narasimhan filtration of $E$, then define $\mu_{\rm max}(E)\,:=\, \mu(E_1)$ and
$\mu_{\rm min}(E)\, =\, \mu(E/E_{n-1})$ \cite{HL}. The subbundle $E_1\, \subseteq\, E$ is called the maximal
semistable subsheaf of $E$.

Let $X$ and $Y$ be irreducible smooth projective curves and
\begin{equation}\label{f1}
f\, :\, X\, \longrightarrow\, Y
\end{equation}
a dominant generically smooth morphism. It is straight-forward to check that
\begin{equation}\label{e0}
\mu_{\rm max}(f_*{\mathcal O}_X)\, =\, 0.
\end{equation}
Indeed, $\mu_{\rm max}(f_*
{\mathcal O}_X)\, \leq\, 0$ because $\text{degree}({\mathcal O}_X)\,=\, 0$ \cite[p.~12824, Lemma 2.2]{BP3}. On the other
hand, we have ${\mathcal O}_Y\, \subset\, f_*{\mathcal O}_X$, which implies that $\mu_{\rm max}(f_*{\mathcal O}_X)\, \geq\,
0$, and thus \eqref{e0} holds.

The following proposition was proved in \cite{BP3}.

\begin{proposition}[{\cite[p.~12828, Proposition 2.6]{BP3} and \cite[p.~12830, Lemma 3.1]{BP3}}]\label{prop-gr}
The following five statements are equivalent:
\begin{enumerate}
\item The maximal semistable subsheaf of $f_*{\mathcal O}_X$ is ${\mathcal O}_Y$.

\item $\dim H^0(X,\, f^*f_* {\mathcal O}_X)\,=\, 1$.

\item The fiber product $X\times_Y X$ is connected.

\item The homomorphism of \'etale fundamental groups
$f_*\, :\, \pi^{\rm et}_1(X)\, \longrightarrow\,\pi^{\rm et}_1(Y)$
induced by $f$ is surjective.

\item The map $f$ does not factor through any nontrivial finite \'etale covering of $Y$.
\end{enumerate}
\end{proposition}

Any morphism $f$ as in \eqref{f1} is called \textit{genuinely ramified}
if the (equivalent) statements in Proposition \ref{prop-gr} hold \cite[p.~12828, Definition 2.5]{BP3}.

\begin{proposition}\label{prop1}
Let $f\, :\, X\, \longrightarrow\, Y$ be a genuinely ramified morphism of smooth
projective curves. Then the vector bundle $((f_*{\mathcal O}_X)/{\mathcal O}_Y)^*$ is ample.
\end{proposition}

\begin{proof}
Since $f$ is genuinely ramified, from Proposition \ref{prop-gr} it follows that 
$$\mu_{\rm max}((f_*{\mathcal O}_X)/{\mathcal O}_Y)\, <\, 0,$$ and hence
we have
\begin{equation}\label{e2}
\mu_{\rm min}(((f_*{\mathcal O}_X)/{\mathcal O}_Y)^*)\,=\, - \mu_{\rm max}((f_*{\mathcal O}_X)/{\mathcal O}_Y)\, > \, 0\, .
\end{equation}

When the characteristic of $k$ is zero, a vector bundle $W$ on $Y$ is ample if and only if the degree
of every nonzero quotient of $W$ is positive \cite[p.~84, Theorem 2.4]{Ha2}. Therefore, from \eqref{e2}
we conclude that $((f_*{\mathcal O}_X)/{\mathcal O}_Y)^*$ is ample, when the characteristic of $k$ is zero. However,
this characterization of ample bundles fails when the characteristic of $k$ is positive (see
\cite[Section~3]{Ha2} for such examples).

We will inductive construct a sequence of vector bundles $\{V_i\}_{i \geq 0}$ on $Y$. First set
$V_0\, =\, {\mathcal O}_Y$. For any $i\, \geq\, i$, let $V_i\,=\, f_*f^*V_{i-1}$. Since we have
$${\mathcal O}_Y\, \subset\, V_1\,=\, f_*f^*{\mathcal O}_Y\,=\, f_*{\mathcal O}_X,$$ it can be
deduced that
\begin{equation}\label{f2}
{\mathcal O}_Y\, \subset\, V_i
\end{equation}
for all $i\, \geq\, 0$. Indeed, this follows inductively,
as the inclusion map ${\mathcal O}_Y\, \hookrightarrow\, V_j$ produces
$$
{\mathcal O}_Y \, \subset\, f_*{\mathcal O}_X \,=\, f_* f^*{\mathcal O}_Y \, \hookrightarrow\, f^*f_*V_j\,=\, V_{j+1}.
$$
This proves \eqref{f2} inductively.

Next we will show that the subsheaf ${\mathcal O}_Y$ in \eqref{f2} is the maximal semistable subsheaf of $V_i$.
This will also be proved using an inductive argument.

First, ${\mathcal O}_Y$ is obviously the maximal semistable subsheaf of $V_0$. Next, from Proposition
\ref{prop-gr} we know that ${\mathcal O}_Y$ is the maximal semistable subsheaf of $V_1$ (recall that $f$
is genuinely ramified). Let
$$
{\mathcal O}_Y \,=\, E^1_1\, \subset\, E^1_2\, \subset\, \cdots\, \subset\, E^1_{n_1-1}\, \subset\, E^1_{n_1}\,=\, V_1
$$
be the Harder--Narasimhan filtration of $V_1$. Since $f^*W$ is semistable if $W$ is so (see \cite[pp. 12823--12824,
Remark 2.1]{BP3}), we conclude that
\begin{equation}\label{e1}
{\mathcal O}_X \,=\, f^*E^1_1\, \subset\, \cdots\, \subset\, f^*E^1_{n_1-1}\, \subset\, f^*E^1_{n_1}\,=\, f^*V_1
\end{equation}
is the Harder--Narasimhan filtration of $f^*V_1$.

For any vector bundle $B$ on $X$, we have $\mu_{\rm max}(f_* B) \, \leq\, \mu_{\rm max}(B)/\text{degree}(f)$
\cite[Lemma 2.2, p.~12824]{BP3}. In view of the Harder--Narasimhan filtration in \eqref{e1}, this implies that
$$
\mu_{\rm max}((f_*f^*E^1_{j+1})/(f_*f^*E^1_{j})) \, <\, 0
$$
for all $1\,\leq\, j\, \, \leq\, n_1-1$, because $\mu_{\rm max}((f^*E^1_{j+1})/(f^*E^1_{j})) \, <\, 0$. Also, as noted before,
the maximal semistable subsheaf of $f_*{\mathcal O}_X$ is ${\mathcal O}_Y$. Combining these we conclude that
${\mathcal O}_Y$ is the maximal semistable subsheaf of $f_*f^*V_1\,=\, V_2$.

The above argument works inductively. To explain this, let
$$
{\mathcal O}_Y \,=\, E^\ell_1\, \subset\, E^\ell_2\, \subset\, \cdots\, \subset\, E^\ell_{n_\ell-1}\, \subset\, E^\ell_{n_\ell}
\,=\, V_\ell
$$
be the Harder--Narasimhan filtration of $V_\ell$. As before, we have
$$
\mu_{\rm max}((f_*f^*E^\ell_{j+1})/(f_*f^*E^\ell_{j})) \, <\, 0
$$
for all $1\,\leq\, j\, \, \leq\, n_\ell-1$, because $\mu_{\rm max}((f^*E^\ell_{j+1})/(f^*E^\ell_{j})) \, <\, 0$. Using this
together with the fact that the maximal semistable subsheaf of $f_*{\mathcal O}_X$ is ${\mathcal O}_Y$
we conclude that ${\mathcal O}_Y$ is the maximal semistable subsheaf of $f_*f^*V_\ell\,=\, V_{\ell+1}$.

The projection formula (see \cite[p.~124, Ch.~II, Ex.~5.1(d)]{Ha3}, \cite{Se}) gives that $V_{i+1} \,=\, f_*f^*V_i\,=\,
V_i\otimes (f_*{\mathcal O}_X)$ for all $i\, \geq\, 1$. This implies that
\begin{equation}\label{e3}
V_i\,=\, (f_*{\mathcal O}_X)^{\otimes i}\,=\, V^{\otimes i}_1
\end{equation}
for all $i\, \geq\, 1$.

Now we assume that the characteristic of $k$ is positive (recall that the proposition was 
proved when the characteristic of $k$ is zero). Let $p$ be the characteristic of $k$. Let 
$$F_Y\, :\, Y\, \longrightarrow\, Y$$ be the absolute Frobenius morphism of $Y$. For any 
vector bundle $W$ on $Y$, we have the inclusion
$$
F^*_Y W\, \subset\, W^{\otimes p}\, ;
$$
it is constructed using the map $W\, \longrightarrow\, W^{\otimes p}$ defined by $v\, \longmapsto\, v^{\otimes p}$.
Therefore, from \eqref{e3} we have
\begin{equation}\label{e4}
(F^n_Y)^* V_1\, \subset\, (V_1)^{\otimes np}\,=\, V_{np}
\end{equation}
for all $n\, \geq\, 1$. Since ${\mathcal O}_Y$ in \eqref{f2} is the maximal semistable subsheaf of $V_i$, from \eqref{e4} we have
$$
(F^n_Y)^* (V_1/{\mathcal O}_Y)\,=\, ((F^n_Y)^* V_1)/{\mathcal O}_Y
\, \subset\, V_{np}/{\mathcal O}_Y,
$$
and
\begin{equation}\label{e5}
\mu_{\rm max}((F^n_Y)^* (V_1/{\mathcal O}_Y)) \, <\, 0,
\end{equation}
because $\mu_{\rm max}(V_{np}/{\mathcal O}_Y) \, <\, 0$.

{}From \eqref{e5} it follows that
$$
\mu_{\rm min}((F^n_Y)^* (V_1/{\mathcal O}_Y)^*) \,=\, - \mu_{\rm max}((F^n_Y)^* (V_1/{\mathcal O}_Y)) \, > \, 0
$$
for all $n\, \geq\, 1$. This implies that $(V_1/{\mathcal O}_Y)^*\,=\, ((f_*{\mathcal O}_X)/{\mathcal O}_Y)^*$
is ample \cite[p.~542, Theorem 2.2]{Bi}.
\end{proof}

\section{Virtual global generation}\label{se3}

Let $E$ be a vector bundle on an irreducible smooth projective curve $Z$. It will be called \textit{virtually
globally generated} if there is a finite surjective morphism
$$
\phi\, :\,M\, \longrightarrow\, Z
$$
from an irreducible smooth projective curve $M$ such that $\phi^*E$ is generated by its global sections. The vector
bundle $E$ is called \textit{\'etale trivializable} if there is a pair $(M,\, \phi)$ as above such that $\phi$ is
\'etale and $\phi^*E$ is trivializable.

If $\text{degree}(E)\, <\, 0$, then $E$ is not virtually globally generated. More generally, $E$ is 
not virtually globally generated if it admits a quotient of negative degree. To give a nontrivial example
of vector bundle which is not virtually globally generated, let $Z$ be a compact connected Riemann surface of genus $g$,
with $g\, \geq\, 2$. Note that the free group of $g$ generators is a quotient of
$\pi_1(Z)$. To see this, express $\pi_1(Z)$ as the quotient of the free group, with generators $a_1,\, \cdots,\, a_g,\, b_1,\,
\cdots,\, b_g$, by the single relation $\prod_{i=1}^g [a_i,\, b_i]\,=\, 1$. Then the quotient of $\pi_1(Z)$ by the normal
subgroup generated by $b_1,\, \cdots,\, b_g$ is the free group generated by $a_1,\, \cdots,\, a_g$. Therefore, there is homomorphism
$$
\rho\, \, :\,\, \pi_1(Z)\,\, \longrightarrow\,\, \text{U}(r),
$$
where $\text{U}(r)$ is the group of $r\times r$ unitary matrices, such that $\rho(\pi_1(Z))$ is a dense subgroup of $\text{U}(r)$ (the
subgroup of $\text{U}(r)$ generated by two general elements of it is dense in ${\rm U}(r)$). Let $E$ denote the flat unitary vector bundle
on $Z$ given by $\rho$. This vector bundle $E$ is stable of degree zero \cite{NS}. Let $M$ be a compact connected Riemann surface and 
$$
\phi\, :\,M\, \longrightarrow\, Z
$$
a surjective holomorphic map. Since the image of the induced homomorphism
$$
\phi_*\,\, :\,\,\pi_1(M)\,\, \longrightarrow\,\, \pi_1(Z)
$$
is a subgroup of $\pi_1(Z)$ of finite index, the image of the following composition of homomorphisms
$$
\pi_1(M)\, \,\stackrel{\phi_*}{\longrightarrow}\, \,\pi_1(Z) \,\, \stackrel{\rho}{\longrightarrow}\,\, \text{U}(r)
$$
is a dense subgroup of $\text{U}(r)$. This implies that $\phi^*E$ is a stable vector bundle of degree zero
\cite{NS}. In particular, we have
$$
H^0(M,\, \phi^*E)\,=\, 0.
$$
Hence $E$ is not virtually globally generated.

\begin{theorem}\label{thm1}
Let $X$ and $Y$ be irreducible smooth projective curves over $k$ and
$$
f\, :\, X\, \longrightarrow\, Y
$$
a generically smooth morphism. Then $f_*{\mathcal O}_X$ fits in a short exact sequence of vector bundles on $Y$
$$
0\, \longrightarrow\, E \, \longrightarrow\, f_*{\mathcal O}_X \, 
\longrightarrow\, V \, \longrightarrow\, 0\, ,
$$
where $E$ is \'etale trivializable and $V^*$ is ample.
\end{theorem}

\begin{proof}
Let
\begin{equation}\label{g1}
S^f\, \subset\,f_*{\mathcal O}_X
\end{equation}
be the maximal semistable subbundle. From \eqref{e0} we know that $\text{degree}(S^f)\,=\, 0$.

The algebra structure of ${\mathcal O}_X$ produces an algebra structure on the direct image $f_*{\mathcal O}_X$.
The subsheaf $S^f$ in \eqref{g1} is a subalgebra. Moreover, there is an \'etale covering $g\, :\, Z\, \longrightarrow\, Y$
such that
\begin{itemize}
\item $f$ factors through $g$, meaning there is a morphism
\begin{equation}\label{g3}
h\, :\, X\, \longrightarrow\, Z
\end{equation}
such that $g\circ h\,=\, f$, and

\item the subsheaf $g_*{\mathcal O}_Z \, \subset\,f_*{\mathcal O}_X$ coincides with $S^f$.
\end{itemize}
(See the proof of \cite[p.~12828, Proposition 2.6]{BP3} and \cite[p.~12829, (2.13)]{BP3}.) Moreover, the map $h$ in \eqref{g3}
is genuinely ramified \cite[p.~12829, Corollary 2.7]{BP3}.

Consider the short exact sequence of vector bundles on $Y$
\begin{equation}\label{g4}
0\, \longrightarrow\, S^f \, \longrightarrow\, f_*{\mathcal O}_X \, \longrightarrow\, Q\,:=\,(f_*{\mathcal O}_X)/S^f
\, \longrightarrow\, 0\, .
\end{equation}
The pullback $g^*Q$, where $Q$ is the vector bundle in \eqref{g4}, is identified with $(h_*{\mathcal O}_X)/{\mathcal O}_Z$,
where $h$ is the map in \eqref{g3}. From Proposition \ref{prop1} we know that $((h_*{\mathcal O}_X)/{\mathcal O}_Z)^*$ is
ample, Since $((h_*{\mathcal O}_X)/{\mathcal O}_Z)^*\,=\, g^*Q^*$, this implies that $Q^*$ in \eqref{g4} is ample
(see \cite[p.~73, Proposition 4.3]{Ha1}).

Since $Q^*$ is ample, in view of \eqref{g4}, it suffices to prove that $S^f$ is a finite vector bundle.

Fix an \'etale Galois covering $\varphi\, :\, M\, \longrightarrow\, Y$ that dominates $g$. In other words, there is a morphism
$$
\beta \, :\, M\, \longrightarrow\, Z
$$
such that $g\circ\beta\,=\, \varphi$. Since $\varphi$ is an \'etale Galois covering, the vector bundle
$\varphi^*\varphi_* {\mathcal O}_M$ is trivializable. On the other hand,
$$
S^f\,=\, g_*{\mathcal O}_Z\, \subset\, \varphi_*{\mathcal O}_M ,
$$
and $S^f$ is a subbundle of $\varphi_*{\mathcal O}_M$. Consider the subbundle
\begin{equation}\label{g5}
\varphi^*S^f\, \subset\, \varphi^*\varphi_* {\mathcal O}_M\, .
\end{equation}
We have $\text{degree}(\varphi^*S^f)\,=\, 0$, because $\text{degree}(S^f)\,=\, 0$, and we also know that
$\varphi^*\varphi_* {\mathcal O}_M$ is trivializable. Consequently, the subbundle $\varphi^*S^f$ in \eqref{g5}
is also trivializable. Hence $S^f$ is \'etale trivializable.
\end{proof}

\begin{corollary}\label{cor2}
Let $f\, :\, X\, \longrightarrow\, Y$ be a generically smooth morphism between two
irreducible smooth projective curves. Then $f$ is genuinely ramified if and only if
$((f_*{\mathcal O}_X)/{\mathcal O}_Y)^*$ is ample.
\end{corollary}

\begin{proof}
In view of Proposition \ref{prop1} it suffices to show that $((f_*{\mathcal O}_X)/
{\mathcal O}_Y)^*$ is not ample if $f$ is not genuinely ramified. If
$f$ is not genuinely ramified, then $\text{rank}(S^f)\, \geq\, 2$ (see \eqref{g1}).
Hence $(S^f/{\mathcal O}_Y)^*$ is a quotient of $((f_*{\mathcal O}_X)/{\mathcal O}_Y)^*$
(see \eqref{g4}). But $\text{degree}((S^f/{\mathcal O}_Y)^*)\,=\, 0$ because
$\text{degree}((S^f)\,=\, 0$. Now $((f_*{\mathcal O}_X)/
{\mathcal O}_Y)^*$ is not ample because its quotient $(S^f/{\mathcal O}_Y)^*$
is not ample.
\end{proof}

\begin{theorem}\label{thm2}
Let $X$ and $Y$ be irreducible smooth projective curves and
$$
f\, :\, X\, \longrightarrow\, Y
$$
a generically smooth morphism. Then $(f_*{\mathcal O}_X)^*$ is virtually globally generated.
\end{theorem}

\begin{proof}
First assume that the characteristic of $k$ is zero. We will show that the short exact sequence
in \eqref{g4} splits. First, the inclusion map ${\mathcal O}_Z \, \hookrightarrow\, h_*{\mathcal O}_X$
splits naturally, where $h$ is the map in \eqref{g3}; in other words,
$$
h_*{\mathcal O}_X\,=\, {\mathcal O}_Z\oplus F\, ;
$$
the fiber of $F$ over any $z\, \in\, Z$ is the space of functions on $h^{-1}(z)$ whose sum is zero. Now
we have
\begin{equation}\label{g6}
f_*{\mathcal O}_X\,=\, g_*h_*{\mathcal O}_X\,=\, g_*({\mathcal O}_Z\oplus F)\,=\, 
(g_*{\mathcal O}_Z)\oplus g_* F\,=\, S^f\oplus g_*F\, .
\end{equation}
{}From \eqref{g4} and \eqref{g6} it follows that the vector bundle $g_*F$ is isomorphic to $Q$. Therefore, from \eqref{g6}
we have
\begin{equation}\label{g7}
(f_*{\mathcal O}_X)^*\,=\, (S^f)^*\oplus Q^*\,.
\end{equation}

Now $(S^f)^*$ is virtually globally generated because $S^f$ is \'etale trivializable, and $Q^*$ is
virtually globally generated because $Q^*$ is ample by Theorem \ref{thm1} (see \cite[p.~46, Theorem 3.6]{BP2}).
Therefore, from \eqref{g7} it follows that $(f_*{\mathcal O}_X)^*$ is virtually globally generated.

Next assume that the characteristic of $k$ is positive. As before, $$F_Y\, :\, Y\, \longrightarrow\, Y$$ is the absolute
Frobenius morphism of $Y$. Consider the exact sequence in \eqref{g4}; recall that $S^f$ is the maximal semistable subsheaf of
$f_*{\mathcal O}_X$. Therefore, there is an integer $n_0$ such that for all $n\, \geq\, n_0$, we have
$$
(F^n_Y)^*f_*{\mathcal O}_X\,=\, (F^n_Y)^* S^f\oplus (F^n_Y)^* Q
$$
\cite[p.~356, Proposition 2.1]{BP1}. Therefore,
\begin{equation}\label{g8}
(F^n_Y)^*(f_*{\mathcal O}_X)^*\,=\, (F^n_Y)^* (S^f)^*\oplus (F^n_Y)^* Q^*\, .
\end{equation}
Now $(F^n_Y)^* (S^f)^*$ is virtually globally generated because $S^f$ is \'etale trivializable and the Frobenius
morphism commutes with \'etale morphisms. Also, $Q^*$ is virtually globally generated because $Q^*$ is ample by Theorem
\ref{thm1} (see \cite[p.~357, Theorem 2.2]{BP1}). Therefore, from \eqref{g8} it follows that
$(f_*{\mathcal O}_X)^*$ is virtually globally generated.
\end{proof}

\begin{corollary}\label{cor1}
Let $X$ and $Y$ be irreducible smooth projective curves and
$$
f\, :\, X\, \longrightarrow\, Y
$$
a generically smooth morphism. Then the following statements hold:
\begin{itemize}
\item If the characteristic of $k$ is zero, then 
$$
(f_*{\mathcal O}_X)^*\,=\, E\oplus A\, ,
$$
where $E$ is \'etale trivializable and $A$ is ample.

\item If the characteristic of $k$ is positive, then there is an integer $n$ such that
$$
(F^n_Y)^*(f_*{\mathcal O}_X)^*\,=\, E\oplus A\, ,
$$
where $E$ is \'etale trivializable and $A$ is ample.
\end{itemize}
\end{corollary}

\begin{proof}
In view of Theorem \ref{thm2}, this follows immediately from \cite[p.~40, Theorem 1.1]{BP2}.
\end{proof}

\begin{corollary}\label{cor3}
Let $X$ and $Y$ be irreducible smooth projective curves and
$$
f\, :\, X\, \longrightarrow\, Y
$$
a generically smooth morphism. Then $((f_*{\mathcal O}_X)/{\mathcal O}_Y)^*$ is
virtually globally generated.
\end{corollary}

\begin{proof}
{}From Theorem \ref{thm2} we know that there is a finite surjective map
$$
\phi\, :\,M\, \longrightarrow\, Y
$$
such that $\phi^*(f_*{\mathcal O}_X)^*$ is generated by its global sections. We have
the short exact sequence of vector bundles on $M$
\begin{equation}\label{z1}
0\, \longrightarrow\, \phi^*((f_*{\mathcal O}_X)/{\mathcal O}_Y)^* \, \longrightarrow\,
\phi^*(f_*{\mathcal O}_X)^* \, \longrightarrow\, \phi^*({\mathcal O}_Y)^*\,=\,
{\mathcal O}_M \, \longrightarrow\, 0\, .
\end{equation}
Since $\phi^*(f_*{\mathcal O}_X)^*$ is generated by its global sections, it has a section
that projects to a nonzero section of ${\mathcal O}_M$. Choosing such a section we obtain
a splitting of \eqref{z1}. Since $\phi^*(f_*{\mathcal O}_X)^*$ is generated by
its global sections, its direct summand $\phi^*((f_*{\mathcal O}_X)/{\mathcal O}_Y)^*$
is also generated by its global sections.
\end{proof}

\begin{remark}\label{r-f}
Corollary \ref{cor3} is not valid in higher dimensions. To give an example, let $X$ denote
${\mathbb C}{\mathbb P}^2$ blown up at the point $(1,\, 0,\, 0)$. The involution of 
${\mathbb C}{\mathbb P}^2$ defined by $(x,\, y,\, z)\, \longmapsto\, (x,\,-y,\, -z)$ lifts
to $X$; let
$$
\tau\, :\, X\, \longrightarrow\, X
$$
be this lifted involution. Set $Y\,:=\, X/({\mathbb Z}/2{\mathbb Z})$ to be the quotient of $X$
for the action of ${\mathbb Z}/2{\mathbb Z}$ given by $\tau$. Let
$$
f\, :\, X\, \longrightarrow\, X/({\mathbb Z}/2{\mathbb Z})\,=\, Y
$$
be the quotient map. Then the line bundle $((f_*{\mathcal O}_X)/{\mathcal O}_Y)^*$ is not virtually globally
generated. To see this, first note that the line bundle $f^*(((f_*{\mathcal O}_X)/{\mathcal O}_Y)^*)$ is
virtually globally generated if $((f_*{\mathcal O}_X)/{\mathcal O}_Y)^*$ is virtually globally generated. But
$$f^*(((f_*{\mathcal O}_X)/{\mathcal O}_Y)^*)\, =\, {\mathcal O}_{X}(D_e+D_\infty)\, ,$$
where $D_e\, \subset\, X$ is the exceptional divisor and $D_\infty\, \subset\, X$ is the inverse image of
$$
\{(0,\, y,\, z)\,\in\, {\mathbb C}{\mathbb P}^2\,\, \mid\,\, y,\, z\, \in\, {\mathbb C}\}\, \subset\,
{\mathbb C}{\mathbb P}^2\, .
$$
It is easy to see that ${\mathcal O}_{X}(D_e+D_\infty)$ is not virtually globally generated. Indeed, if
$$\varpi\, :\, Z\, \longrightarrow\, X$$ is a finite surjective proper map, then every section of
$\varpi^*{\mathcal O}_{X}(D_e+D_\infty)$ vanishes on $\varpi^{-1}(D_e)$.
\end{remark}

\section*{Acknowledgements}

The first author is partially supported by a J. C. Bose Fellowship (JBR/2023/000003).

\medskip
\noindent
\textbf{Competing interests:}\, The authors declare none.


\begin{thebibliography}{ZZZZ}

\bibitem[Bi]{Bi} I. Biswas, A criterion for ample vector bundles over a curve in positive
characteristic, {\it Bull. Sci. Math.} {\bf 129} (2005), 539--543. 

\bibitem[BP1]{BP1} I. Biswas and A. J. Parameswaran, On the ample vector bundles over curves in
positive characteristic, {\it Comp. Ren. Math. Acad. Sci. Paris} {\bf 339} (2004), 355--358. 

\bibitem[BP2]{BP2} I. Biswas and A. J. Parameswaran, A criterion for virtual global generation,
{\it Ann. Sc. Norm. Super. Pisa Cl. Sci.} {\bf 5} (2006), 39--53.

\bibitem[BP3]{BP3} I. Biswas and A. J. Parameswaran, Ramified covering maps and
stability of pulled back bundles, {\it Int. Math. Res. Not.}, Vol. 2022, Issue 17, 12821--12851.

\bibitem[CLV]{CLV} I. Coskun, E. Larson and I. Vogt, Stability of Tschirnhausen bundles, 
arXiv:2207.07257.

\bibitem[Ha1]{Ha1} R. Hartshorne, Ample vector bundles on curves, Ample vector bundles,
{\it Inst. Hautes \'Etudes Sci. Publ. Math.} {\bf 29} (1966), 63--94. 

\bibitem[Ha2]{Ha2} R. Hartshorne, Ample vector bundles on curves, {\it Nagoya Math. Jour.} {\bf 43} (1971), 73--89.

\bibitem[Ha3]{Ha3} R. Hartshorne, {\it Algebraic geometry}, Graduate Texts in Mathematics, No. 52. Springer-Ve>
New York-Heidelberg, 1977.

\bibitem[HL]{HL} D. Huybrechts and M. Lehn, {\it The geometry of moduli spaces of sheaves}, Aspects
of Mathematics, E31, Friedr. Vieweg~\&~Sohn, Braunschweig, 1997.

\bibitem[NS]{NS} M. S. Narasimhan and C. S. Seshadri, Stable and unitary vector bundles on a compact
Riemann surface, {\it Ann. of Math.} {\bf 82} (1965), 540--567.

\bibitem[Se]{Se} J.-P. Serre, G\'eom\'etrie alg\'ebrique et g\'eom\'etrie analytique,
{\it Ann. Inst. Fourier} {\bf 6} (1956), 1--42.

\end{thebibliography}
\end{document}